\documentclass[12pt]{article}
\usepackage{color, amsmath,amssymb, amsfonts, amstext,amsthm, latexsym}
\usepackage[active]{srcltx}
\usepackage{extarrows}
\usepackage{amsfonts,amsmath,amsthm, amssymb}
\usepackage{latexsym, euscript, epic, eepic}
\usepackage{lineno}

\newtheorem{cro}{Corollary}[section]
\newtheorem{defn}{Definition}[section]

\newtheorem{thm}{Theorem}[section]
\newtheorem{lem}{Lemma}[section]

\begin{document}

\title{The Historic Set of Ergodic Averages in Some
Nonuniformly Hyperbolic Systems
 \footnotetext {* Corresponding author}
  }
\author{Zheng  Yin$^1$, Ercai Chen$^{1,2*}$, Xiaoyao Zhou$^{3}$ \\
 \small   1 School of Mathematical Sciences and Institute of Mathematics, Nanjing Normal University,\\
  \small   Nanjing 210023, Jiangsu, P.R.China\\
   \small   2 Center of Nonlinear Science, Nanjing University,\\
    \small   Nanjing 210093, Jiangsu, P.R.China\\
     \small  3 Department of Mathematics, University of Science and Technology of China,\\
      \small  Hefei, Anhui,230026, P.R.China\\
      \small    e-mail: zhengyinmail@126.com,\\
       \small    ecchen@njnu.edu.cn,\\
        \small    zhouxiaoyaodeyouxian@126.com\\
}
\date{}
\maketitle

\begin{center}
 \begin{minipage}{120mm}
{\small {\bf Abstract.} This article is devoted to the study of the historic set of  ergodic averages in some nonuniformly
hyperbolic systems. In particular, our results hold for the robust
classes of multidimensional nonuniformly expanding local
diffeomorphisms and Viana maps.}
\end{minipage}
 \end{center}

\vskip0.5cm
\section{Introduction and Preliminaries}
$(X,d,T)$ (or $(X,T)$ for short) is a topological dynamical system means that $(X,d)$ is a compact metric space together with a continuous self-map $T:X\to X.$
For a continuous function $\varphi:X\to\mathbb R, X$ can be divided into the following two parts:
\begin{align*}
X&=\bigcup\limits_{\alpha\in\mathbb R}\left\{x\in X:\lim\limits_{n\to\infty}\frac{1}{n}\sum\limits_{i=0}^{n-1}\varphi(T^ix)=\alpha\right\}\cup\left\{x\in X:\lim\limits_{n\to\infty}\frac{1}{n}\sum\limits_{i=0}^{n-1}\varphi(T^ix)\text{ does not exist }\right\}.
\end{align*}
The level set $\left\{x\in X:\lim\limits_{n\to\infty}\frac{1}{n}\sum\limits_{i=0}^{n-1}\varphi(T^ix)=\alpha\right\}$ is so-called multifractal decomposition sets of  ergodic averages of $\varphi$ in multifractal analysis.
There are
fruitful results about the descriptions of the structure (Hausdorff
dimension or topological entropy  or topological pressure) of these
level sets  in topological dynamical systems. Early
studies of the level sets was about their dimensions and topological
entropy. See Barreira \& Saussol \cite{BarSau}, Barreira, Saussol \&
Schmeling \cite{BarSauSch},
Olsen \cite{Ols}, Olsen \& Winter
\cite{OlsWin}, Takens \& Verbitskiy \cite{TakVer}, Zhou,
Chen \& Cheng \cite{ZhoCheChe} and Pfister \& Sullivan
\cite{PfiSul2}. Recently, the topological pressures of the level
sets has also been investigated. See Thompson \cite{Tho2}, Pei \&
Chen \cite{PeiChe1}
 and Zhou \& Chen \cite{ZhoChe}.

The set $\widehat{X}(\varphi,T):=\left\{x\in X:\lim\limits_{n\to\infty}\frac{1}{n}\sum\limits_{i=0}^{n-1}\varphi(T^ix)\text{ does not exist }\right\}$ is called the historic set of  ergodic averages of $\varphi.$ This terminology was introduced by Ruelle in \cite{Rue}. It is also called  non-typical points (see \cite{BarSch}), irregular set (see \cite{Tho1,Tho3}) and divergence points (see \cite{CheKupShu,Ols,OlsWin}).
 If this limit $\lim\limits_{n\to\infty}\frac{1}{n}\sum\limits_{i=0}^{n-1}\varphi(T^ix)$ does not exist, it follows that `partial averages' $\lim\limits_{n\to\infty}\frac{1}{n}\sum\limits_{i=0}^{n-1}\varphi(T^ix)$ keep change considerably so that their values  give information about the epoch to which $n$ belongs. The problem, whether there are persistent classes of smooth dynamical systems such that the set of initial states which
give rise to orbits with historic behavior has ¡®positive Lebesgue measure¡¯ was discussed by Ruelle in \cite{Rue} and  Takens
in \cite{Tak}.
 By Birkhoff's ergodic theorem, $\widehat{X}(\varphi,T)$ is not detectable from the point of view of an invariant measure, i.e., for any invariant measure $\mu,$
\begin{align*}
\mu(\widehat{X}(\varphi,T))=0.
\end{align*}
Hence, at first, the set $\widehat{X}(\varphi,T)$ has until recently been considered of little interest in dynamical systems and geometric measure theory.
However, recent work \cite{CheXio,FanFenWu,FenLauWu} has changed such attitudes. They have shown that in many cases the set can have full Hausdorff dimension, i.e.,
\begin{align*}
\dim_H(\widehat{X}(\varphi,T))=\dim_H(X).
\end{align*}
Barreira and Schmeling \cite{BarSch}  confirmed this in the uniformly hyperbolic setting in symbolic dynamics.  In 2005,  Chen, Kupper and Shu \cite{CheKupShu} proved that $\widehat{X}(\varphi,T)$
is either
empty or carries full entropy for maps with the specification property.  Thompson \cite{Tho1} extended it to topological
pressure
for maps with the specification property.
Zhou and Chen \cite{ZhoChe} also investigated  the multifractal analysis for the
historic set in topological dynamical systems  with g-almost product property.

Now, nonuniformly hyperbolic systems attract more and more
attentions. We refer the readers to  Barreira \& Pesin
\cite{BarPes}, Chung \&  Takahasi \cite{ChuTak},   Johansson,
Jordan,  Oberg \& Pollicott \cite{JohJorObePol},   Jordan \&   Rams
\cite{JorRam}, Liang, Liao, Sun \& Tian \cite{LiaLiaSUnTia},
Oliveira \cite{Oliv1}, Oliveira \& Viana \cite{OliVia}
and references therein for recent results in
nonuniformly hyperbolic systems. It is well known that the
specification property plays an important role in some uniformly
hyperbolic dynamical systems. The notion of specification is
slightly weaker than the one introduced by Bowen that requires any
finite sequence of pieces of orbit is well approximated by periodic
orbits. It implies that the dynamical systems have some mixing
property. One should mention that other mild forms of specification
were introduced by
 Pfister \& Sullivan \cite{PfiSul2} and Thompson \cite{Tho3} to the
study of multifractal formalism for Birkhoff averages associated to
beta-shifts, and by Pfister \& Sullivan \cite{PfiSul1}, Yamamoto
\cite{Yam2} and Varandas \cite{Var} to study large deviations.  This
article will use the weak form of specification introduced by
Varandas \cite{Var} in a nonuniformly hyperbolic context.

 Denote by $M(X)$ and $M(X,T)$ the set of all Borel probability
measures on $X$ and the collection of all $T$-invariant Borel
probability measures, respectively. It is well known that $M(X)$ and
$M(X,T)$  equipped with weak* topology are both convex, compact
spaces.

\begin{defn}{\rm\cite{Var}}
We say that $(T,m)$ satisfies the non-uniform specification property with a time lag of $p(x,n,\epsilon)$
if there exists $\delta>0$ such that for $m$-almost every $x$ and
every $0<\epsilon<\delta$ there exists an integer
$p(x,n,\epsilon)\geq1$ satisfying
\begin{align*}
 \lim\limits_{\epsilon\to0}\limsup\limits_{n\to\infty}\frac{1}{n}p(x,n,\epsilon)=0
  \end{align*}
and so that the following holds: given points $x_1,\cdots, x_k$ in a
full $m$-measure set and positive integers $n_1,\cdots,n_k,$ if
$p_i\geq p(x_i,n_i,\epsilon)$ then there exists $z$ that
$\epsilon$-shadows the orbits of each $x_i$ during $n_i$ iterates
with a time lag of $p(x_i,n_i,\epsilon)$ in between $T^{n_i}(x_i)$
and $x_{i+1},$ that is
\begin{align*}
z\in B_{n_1}(x_1,\epsilon) {\rm~ and~}
T^{n_1+p_1+\cdots+n_{i-1}+p_{i-1}}(z)\in B_{n_i}(x_i,\epsilon)
\end{align*}
for every $2\leq i\leq k,$ where
\begin{align*}
B_n(x,\epsilon)=\{y:d_n(x,y)<\epsilon\}:=\left\{y:\max\limits_{0\leq i\leq n-1}\{d(T^ix,T^iy)\}<\epsilon\right\}.
\end{align*}
\end{defn}
We  assume that the shadowing property holds on a set $K$ with a time lag $p(x,n,\epsilon)$
throughout the paper. From definition 1.1, we know
for any $x\in K$ and $0<\epsilon_1<\epsilon_2$, $p(x,n,\epsilon_1)\geq p(x,n,\epsilon_2)$.
Hence for any $x\in K$ and $\epsilon>0$, $\limsup_{n\to\infty}\frac{p(x,n,\epsilon)}{n}=0$.
Obviouly, $K$ is $T$-invariant. Let $M(K,T)$ denote the subset of $M(X,T)$ for which the measures $\mu$ satisfy $\mu(K)=1$
and $E(K,T)$ denote those which are ergodic.

\begin{defn}{\rm\cite{Pes}}
Suppose $Z\subset X$ be an arbitrary Borel set and $\psi\in C(X)$.
Let $\Gamma_{n}(Z,\epsilon)$ be the
collection of all finite or countable covers of $Z$ by sets of the
form $B_{m}(x,\epsilon),$ with $m\geq n$. Let $S_{n}\psi(x):=\sum_{i=0}^{n-1}\psi(T^{i}x)$. Set
\begin{align*}
M(Z,t,\psi,n,\epsilon):=\inf_{\mathcal{C}\in\Gamma_{n}(Z,\epsilon)}\left\{\sum_{B_{m}(x,\epsilon)\in
\mathcal{C}}\exp (-tm+\sup_{y\in
B_{m}(x,\epsilon)}S_{m}\psi(y))\right\},
\end{align*}
and
\begin{align*}
M(Z,t,\psi,\epsilon)=\lim_{n\to\infty}M(Z,t,\psi,n,\epsilon).
\end{align*}
Then there exists a unique number $P(Z,\psi,\epsilon)$ such that
$$P(Z,\psi,\epsilon)=\inf\{t:M(Z,t,\psi,\epsilon)=0\}=\sup\{t:M(Z,t,\psi,\epsilon)=\infty\}.$$
$P(Z,\psi)=\lim_{\epsilon\to0}P(Z,\psi,\epsilon)$ is called
the topological pressure of $Z$ with respect to $\psi$.
\end{defn}

It is  obvious that the following hold:

\begin{enumerate}
  \item[(1)] $P(Z_1,\psi)\leq P(Z_2,\psi)$ for any $Z_1\subset Z_2\subset X$;
  \item[(2)] $P(Z,\psi)=\sup_i P(Z_i,\psi)$, where $Z=\bigcup_i Z_i\subset X$.
\end{enumerate}

Now, we state the main result of this article as follows:
\begin{thm}
Let $(X,T)$ be a topological dynamical system. Assume $(T,m)$ satisfies non-uniform specification
property. Assume that $\varphi\in C(X)$ satisfies $\inf\limits_{\mu\in M(K,T)}\int\varphi d\mu
<\sup\limits_{\mu\in M(K,T)}\int\varphi d\mu$, then  $\widehat{X}(\varphi,T)\neq\emptyset$ and for all $\psi\in C(X)$,
\begin{align*}
P(\widehat{X}(\varphi,T),\psi)\geq\sup\left\{h_\mu+\int\psi d\mu:\mu\in M(K,T)\right\},
\end{align*}
where
\begin{align*}
\widehat{X}(\varphi,T)=\left\{x\in X:\lim\limits_{n\to\infty}\frac{1}{n}\sum\limits_{i=0}^{n-1} \varphi (T^ix) \text{ \rm does not exist}\right\}.
\end{align*}
If there exists $\mu\in M(K,T)$ such that $\mu$ is a equilibrium state for $T$ with respect to potential $\psi$,
then we have $P(\widehat{X}(\varphi,T),\psi)=P(X,\psi)$.
\end{thm}

\begin{thm}
Let us assume the hypotheses of theorem 1.1. Let
\begin{align*}
\mathbf{C}:=\sup\left\{h_\mu+\int\psi d\mu:\mu\in M(K,T)\right\}.
\end{align*}
We assume further that $P(X,\psi)$ is finite and for all $\gamma>0,$ there exist ergodic measures $\mu_1,\mu_2\in M(K,T)$ satisfying
\begin{enumerate}
  \item $h_{\mu_i}+\int\psi d\mu_i>\mathbf{C}-\gamma$ for $i=1,2$,
  \item $\int\varphi d\mu_1\neq \int\varphi d\mu_2$.
\end{enumerate}
Then $P(\widehat{X}(\varphi,T),\psi)\geq\mathbf{C}$. If there exists $\mu\in M(K,T)$ such that $\mu$ is a equilibrium state for $T$ with respect to potential $\psi$, then we have $P(\widehat{X}(\varphi,T),\psi)=P(X,\psi)$.
\end{thm}
At first, we prove Theorem 1.2 and then explain how to remove additional hypothesis to obtain theorem 1.1.
The case that $P(X,\psi)$ is infinite can be included in our proof.

\section{Proof of Main Result}

Before showing the lower bound, we give some important lemmas as
follows.

\begin{lem}{\rm\cite{Tho1,Tho2}}
Let $(X,d)$ be a compact metric space, $T:X\to X$ be a continuous map and $\mu$
be an ergodic invariant measure. For $\epsilon>0$, $\gamma\in (0,1)$ and $\psi\in C(X)$, define
\begin{align*}
N^\mu(\psi,\gamma,\epsilon,n)=\inf\left\{\sum_{x\in S}\exp\left\{\sum_{i=0}^{n-1}\psi(T^ix)\right\}\right\},
\end{align*}
where the infimum is taken over all sets $S$ which $(n,\epsilon)$ span some set $Z$ with $\mu(Z)\geq1-\gamma$.
we have
\begin{align*}
h_\mu+\int\psi d\mu=\lim_{\epsilon\to0}\liminf_{n\to\infty}\frac{1}{n}\log N^\mu(\psi,\gamma,\epsilon,n).
\end{align*}
The formula remains true if we replace the $\liminf$ by $\limsup$.
\end{lem}

\begin{lem}{\rm\cite{Tho1,Tho2}}
(Generalised Pressure Distribution Principle) Let $(X,d,T)$ be a
topological dynamical system. Let $Z\subset X$ be an arbitrary Borel
set. Suppose there exist $\epsilon>0$ and $s\in\mathbb{R}$ such that one can
find a sequence of Borel probability measures $\mu_k,$ a constant
$K>0$ and an integer $N$ satisfying
\begin{align*}
\limsup\limits_{k\to\infty}\mu_k(B_n(x,\epsilon))\leq
K\exp(-ns+\sum\limits_{i=0}^{n-1}\psi(T^ix))
\end{align*}
for every ball $B_n(x,\epsilon)$ such that $B_n(x,\epsilon)\cap
Z\neq\emptyset$ and $n\geq N.$ Furthermore, assume that at least one
limit measure $\nu$ of the sequence $\mu_k$ satisfies $\nu(Z)>0.$
Then $P(Z, \psi,\epsilon)\geq s.$
\end{lem}

Fix a small $\gamma>0,$ and take the measures $\mu_1$ and $\mu_2$ provided by hypothesis. Choose $\delta>0$ so small that
\begin{align*}
\left|\int\varphi d\mu_1-\int\varphi d\mu_2\right|>4\delta.
\end{align*}
Let $\rho:\mathbb N \to\{1,2\}$ be given by $\rho(k)=(k+1)(\text{mod }2)+1.$
For $\epsilon>0$, by Egorov's theorem and Birkhoff's ergodic theorem, we can choose a strictly decreasing sequence
$\delta_k\to0$ with $\delta_1<\delta$ and a strictly increasing sequence $l_k\to\infty$ so the sets
\begin{align}
X_k:=\left\{x\in K:\frac{p(x,n,\epsilon/4)}{n}<\frac{1}{2^k}\text{ for all } n\geq l_k\right\}
\end{align}
and
\begin{align}
Y_k:=\left\{x\in K:\left|\frac{1}{n}S_n\varphi(x)-\int\varphi d\mu_{\rho(k)}\right|<\delta_k\text{ for all } n\geq l_k\right\}
\end{align}
satisfy $\mu_{\rho(k)}(X_k)>1-\frac{\gamma}{2}$ and $\mu_{\rho(k)}(Y_k)>1-\frac{\gamma}{2}$ for every $k$.
Then for any $k\in \mathbb{N}$, $\mu_{\rho(k)}(X_k\cap Y_k)>1-\gamma$.

\begin{lem}
For any sufficiently small $\epsilon>0,$ there exists a sequence $n_k\to\infty$ and a countable collection of finite sets $\Theta_k$ such that each $\Theta_k$ is an $(n_k,4\epsilon)$ separated set for $X_k\cap Y_k$ and $M_k:=\sum\limits_{x\in\Theta_k}\exp\left\{\sum\limits_{i=0}^{n_k-1}\psi(T^ix)\right\}$
satisfying
\begin{align*}
M_k\geq\exp(n_k(\mathbf{C}-4\gamma)).
\end{align*}
Furthermore, the sequence $n_k$ can be chosen so that $n_{k}\geq l_{k}$ and for any $x\in\Theta_{k}$, $\frac{p(x,n_{k},\epsilon/4)}{n_{k}}<\frac{1}{2^{k}}$.
\end{lem}
\begin{proof}
By lemma 2.1, we can choose $\epsilon$ sufficiently small so
\begin{align*}
\liminf_{n\to\infty}\frac{1}{n}\log N^{\mu_i}(\psi,\gamma,4\epsilon,n)\geq h_{\mu_i}+\int\psi d\mu_i\geq\mathbf{C}-2\gamma \text{ for } i=1,2.
\end{align*}
For $k\in\mathbb{N}$, let
\begin{align*}
Q_n(X_k\cap Y_k,\psi,4\epsilon)=\inf\left\{\sum_{x\in S}\exp\left\{\sum_{k=0}^{n-1}\psi(T^ix)\right\}:S \text{ is } (n,4\epsilon) \text{ spanning set for } X_k\cap Y_k\right\},\\
P_n(X_k\cap Y_k,\psi,4\epsilon)=\sup\left\{\sum_{x\in S}\exp\left\{\sum_{k=0}^{n-1}\psi(T^ix)\right\}:S \text{ is } (n,4\epsilon) \text{ separated set for } X_k\cap Y_k\right\}.
\end{align*}
Since $\mu_{\rho(k)}(X_k\cap Y_k)>1-\gamma$ for every $k$, we have
\begin{align*}
P_n(X_k\cap Y_k,\psi,4\epsilon)\geq Q_n(X_k\cap Y_k,\psi,4\epsilon)\geq N^{\mu_i}(\psi,\gamma,4\epsilon,n).
\end{align*}
Let $M(k,n)=P_n(X_k\cap Y_k,\psi,4\epsilon)$. For each $k$, we obtain
\begin{align*}
\liminf_{n\to\infty}\frac{1}{n}\log M(k,n)\geq \liminf_{n\to\infty}\frac{1}{n}\log N^{\mu_{\rho(k)}}(\psi,\gamma,4\epsilon,n)\geq\mathbf{C}-2\gamma.
\end{align*}
We choose a sequence $n_k\to\infty$ such that $n_k\geq l_k$ and
\begin{align*}
\frac{1}{n_k}\log M(k,n_k)\geq\mathbf{C}-3\gamma.
\end{align*}
For each $k$, let $\Theta_k$ be a $(n_k,4\epsilon)$ separated set for $X_k\cap Y_k$ which satisfies
\begin{align*}
\frac{1}{n_k}\log\left\{\sum_{x\in \Theta_k}\exp\left\{\sum_{i=0}^{n_k-1}\psi(T^ix)\right\}\right\}\geq\frac{1}{n_k}\log M(k,n_k)-\gamma.
\end{align*}
Let $M_k:=\sum_{x\in \Theta_k}\exp\left\{\sum_{i=0}^{n_k-1}\psi(T^ix)\right\}$, then
\begin{align*}
\frac{1}{n_k}\log M_k\geq\frac{1}{n_k}\log M(k,n_k)-\gamma\geq\mathbf{C}-4\gamma,
\end{align*}
the desired results follows.
\end{proof}

We choose $\epsilon$ so small that
$\text{Var}(\psi,\epsilon):=\sup\{|\psi(x)-\psi(y)|:d(x,y)<\epsilon\}<\gamma$, and
$\text{Var}(\varphi,\epsilon):=\sup\{|\varphi(x)-\varphi(y)|:d(x,y)<\epsilon\}<\frac{\delta}{4}$.
We fix all the ingredients as before.
\subsubsection{Construction of the fractal F}

Let $n_0=0$. Let us choose a sequence with $N_{0}=0$ and $N_{k}$ increasing to $\infty$ sufficiently quickly so that
\begin{align}
\lim_{k\rightarrow\infty}\frac{n_{k+1}+\max\limits_{x\in\Theta_{k+1}}p(x,n_{k+1},\frac{\epsilon}{4})}{N_{k}}=0,
\end{align}
\begin{align}
\lim_{k\rightarrow\infty}\frac{N_{1}(n_{1}+\max\limits_{x\in\Theta_{1}}p(x,n_{1},\frac{\epsilon}{4}))
+\cdots+N_{k}(n_{k}+\max\limits_{x\in\Theta_{k}}p(x,n_{k},\frac{\epsilon}{4}))}{N_{k+1}}=0.
\end{align}
We enumerate the points in the set $\Theta_{i}$ and consider the set $\Theta_i^{N_{i}}$. Let $\underline{x}_{i}=(x_{1}^{i},\cdots,x_{N_{i}}^{i})\in \Theta_{i}^{N_{i}}$. For any $(\underline{x}_{1},\cdots,\underline{x}_{k})\in \Theta_{1}^{N_{1}}\times\cdots\times \Theta_{k}^{N_{k}}$, by the non-uniform specification property, we have
\begin{align*}
B(\underline{x}_{1},\cdots,\underline{x}_{k})=&
\bigcap_{i=1}^{k}\bigcap_{j=1}^{N_{i}}
T^{-\sum_{l=0}^{i-1}N_{l}(n_{l}+\max\limits_{x\in\Theta_{l}}p(x,n_{l},\frac{\epsilon}{4}))-
(j-1)(n_{i}+\max\limits_{x\in\Theta_{i}}p(x,n_{i},\frac{\epsilon}{4}))}B_{n_{i}}(x_{j}^{i},\frac{\epsilon}{4})\\
\neq&\emptyset.
\end{align*}
We define $F_{k}$ by
\begin{align*}
F_{k}=\bigcup\{\overline{B(\underline{x}_{1},\cdots,\underline{x}_{k})}:(\underline{x}_{1},\cdots,\underline{x}_{k})\in \Theta_{1}^{N_{1}}\times\cdots\times \Theta_{k}^{N_{k}}\}.
\end{align*}
Obviously, $F_{k}$ is compact and $F_{k+1}\subset F_{k}$. Define $F=\cap_{k=1}^{\infty}F_{k}$.

\begin{lem}
For any $p\in F$, $\lim\limits_{k\to\infty}\frac{1}{t_k}\sum\limits_{i=0}^{t_k-1}\varphi(T^i p)$ does not exist, where $t_{k}=\sum_{i=0}^{k}N_{i}(n_{i}+\max\limits_{x\in\Theta_{i}}p(x,n_{i},\frac{\epsilon}{4}))$.
\end{lem}
\begin{proof}
Choose $p\in F$, then for any $k\in \mathbb{N}$, $p\in F_{k}$. Let $p_{k}=T^{t_{k-1}}(p)$.
Then there exists $(x_{1}^{k},\cdots\,x_{N_{k}}^{k})\in \Theta_{k}^{N_{k}}$ such that
\begin{align*}
p_{k}\in\bigcap_{j=1}^{N_{k}}T^{-(j-1)(n_{k}+\max\limits_{x\in\Theta_{k}}p(x,n_{k},\frac{\epsilon}{4}))}\overline{B_{n_{k}}}(x_{j}^{k},\frac{\epsilon}{4}).
\end{align*}
Let $a_{j}=(j-1)(n_{k}+\max\limits_{x\in\Theta_{k}}p(x,n_{k},\frac{\epsilon}{4}))$. We have
\begin{align*}
&\left|S_{N_{k}(n_{k}+\max\limits_{x\in\Theta_{k}}p(x,n_{k},\frac{\epsilon}{4}))}\varphi(p_{k})-N_{k}(n_{k}+\max_{x\in\Theta_{k}}p(x,n_{k},\frac{\epsilon}{4}))\int\varphi d\mu_{\rho(k)}\right|\\
\leq&\left|\sum_{j=1}^{N_{k}}S_{n_{k}}\varphi(T^{a_{j}}p_{k})-N_{k}n_{k}\int\varphi d\mu_{\rho(k)}\right|
+N_{k}\max_{x\in\Theta_{k}}p(x,n_{k},\frac{\epsilon}{4})(\|\varphi\|+\int\varphi d\mu_{\rho(k)})\\
\leq&\sum_{j=1}^{N_{k}}\left|S_{n_{k}}\varphi(T^{a_{j}}p_{k})-S_{n_{k}}\varphi(x_{j}^{k})\right|+
\sum_{j=1}^{N_{k}}\left|S_{n_{k}}\varphi(x_{j}^{k})-n_{k}\int\varphi d\mu_{\rho(k)}\right|\\
+&N_{k}\max_{x\in\Theta_{k}}p(x,n_{k},\frac{\epsilon}{4})(\|\varphi\|+\int\varphi d\mu_{\rho(k)})\\
\leq&n_{k}N_{k}\{\text{Var}(\varphi,\frac{\epsilon}{4})+\delta_{k}\}
+N_{k}\max_{x\in\Theta_{k}}p(x,n_{k},\frac{\epsilon}{4})\{\|\varphi\|+\int\varphi d\mu_{\rho(k)}\}.
\end{align*}
Since $\text{Var}(\varphi,\epsilon)<\frac{\delta}{4}$, then for sufficiently large $k$, we have
\begin{align*}
\left|\frac{S_{N_{k}(n_{k}+\max\limits_{x\in\Theta_{k}}p(x,n_{k},\epsilon/4))}\varphi(p_{k})}{N_{k}(n_{k}+\max\limits_{x\in\Theta_{k}}p(x,n_{k},\epsilon/4))}
-\int\varphi d\mu_{\rho(k)}\right|\leq\frac{\delta}{2}.
\end{align*}
One can readily verify that $\frac{N_{k}(n_{k}+\max\limits_{x\in\Theta_{k}}p(x,n_{k},\epsilon/4))}{t_{k}}\to1$. Thus for sufficiently large $k$, we have $\left|\frac{N_{k}(n_{k}+\max\limits_{x\in\Theta_{k}}p(x,n_{k},\epsilon/4))}{t_{k}}-1\right|
\leq\frac{\delta}{4\|\varphi\|}$. We have
\begin{align*}
&\left|\frac{1}{t_{k}}S_{t_{k}}\varphi(p)-
\frac{S_{N_{k}(n_{k}+\max\limits_{x\in\Theta_{k}}p(x,n_{k},\epsilon/4))}\varphi(p_{k})}{N_{k}(n_{k}+\max\limits_{x\in\Theta_{k}}p(x,n_{k},\epsilon/4))}\right|\\
\leq&\left|\frac{1}{t_{k}}S_{t_{k}-N_{k}(n_{k}+\max\limits_{x\in\Theta_{k}}p(x,n_{k},\epsilon/4))}\varphi(p)\right|+\\
&\left|\frac{S_{N_{k}(n_{k}+\max\limits_{x\in\Theta_{k}}p(x,n_{k},\epsilon/4))}\varphi(p_{k})}{N_{k}(n_{k}+\max\limits_{x\in\Theta_{k}}p(x,n_{k},\epsilon/4))}
(\frac{N_{k}(n_{k}+\max\limits_{x\in\Theta_{k}}p(x,n_{k},\epsilon/4))}{t_{k}}-1)\right|\\
\leq&\frac{\delta}{2}.
\end{align*}
Since for sufficiently large $k$,
\begin{align*}
\left|\frac{1}{t_k}\sum\limits_{i=0}^{t_k-1}\varphi(T^i p)-\int\varphi d\mu_{\rho(k)}\right|\leq\delta
<\frac{\left|\int\varphi d\mu_1-\int\varphi d\mu_2\right|}{4},
\end{align*}
the desired results follows.
\end{proof}
\subsubsection{Construction of a Special Sequence
of Measures $\mu_k$}
For each $\underline{x}=(\underline{x}_{1},\cdots,\underline{x}_{k})\in \Theta_{1}^{N_{1}}\times\cdots\times \Theta_{k}^{N_{k}}$,
we choose one point $z=z(\underline{x})$  such that $z\in B(\underline{x}_{1},\cdots,\underline{x}_{k})$. Let $L_k$
be the set of all points constructed in this way. The following lemma shows that
$\#L_k=\#\Theta_{1}^{N_{1}}\times\cdots\times\#\Theta_{k}^{N_{k}}$.

\begin{lem}
Let $\underline{x}$ and $\underline{y}$ be distinct elements of $\Theta_{1}^{N_{1}}\times\cdots\times \Theta_{k}^{N_{k}}$.
Then $z_1=z(\underline{x})$ and $z_2=z(\underline{y})$ are $(t_k,3\epsilon)$ separated points.
\end{lem}
\begin{proof}
Since $\underline{x}\neq\underline{y}$, there exists $i,j$, such that $x_{j}^{i}\neq y_{j}^{i}$. We have
\begin{align*}
d_{n_{i}}(x_{j}^{i},T^{t_{i-1}+(j-1)(n_{i}+\max\limits_{x\in \Theta_{i}}p(x,n_{i},\frac{\epsilon}{4}))}z_{1})<\frac{\epsilon}{4},\\
d_{n_{i}}(y_{j}^{i},T^{t_{i-1}+(j-1)(n_{i}+\max\limits_{x\in \Theta_{i}}p(x,n_{i},\frac{\epsilon}{4}))}z_{2})<\frac{\epsilon}{4}.
\end{align*}
Together with $d_{n_{i}}(x_{j}^{i},y_{j}^{i})>4\epsilon$, we have
\begin{align*}
d_{t_{k}}(z_{1},z_{2})\geq&d_{n_{i}}(T^{t_{i-1}+(j-1)(n_{i}+\max\limits_{x\in \Theta_{i}}p(x,n_{i},\frac{\epsilon}{4}))}z_{1},
T^{t_{i-1}+(j-1)(n_{i}+\max\limits_{x\in \Theta_{i}}p(x,n_{i},\frac{\epsilon}{4}))}z_{2})\\
\geq&4\epsilon-\frac{\epsilon}{4}-\frac{\epsilon}{4}>3\epsilon.
\end{align*}
\end{proof}

We now define the measures on $F$ which yield the required estimates
for the pressure distribution principle. For each $z\in L_k,$ we
associate a number $\EuScript{L}_k(z)\in (0,\infty).$ Using these
numbers as weights, we define, for each $k,$ an atomic measure
centered on $L_k.$ Precisely, if
$z=z(\underline{x}_1,\cdots,\underline{x}_k),$ we define
\begin{align*}
\EuScript{L}_k(z):=\mathcal{L}(\underline{x}_1)\cdots\mathcal{L}(\underline{x}_k),
\end{align*}
where if
$\underline{x}_i=(x_1^i,\cdots,x^i_{N_i})\in\Theta_{i}^{N_i},$
then
\begin{align*}
\mathcal{L}(\underline{x}_i):=\prod\limits_{l=1}^{N_i}\exp
S_{n_i}\psi(x^i_{l}).
\end{align*}
We define $\nu_k:=\sum\limits_{z\in L_k}\delta_z\EuScript{L}_k(z).$
We normalize $\nu_k$ to obtain a sequence of probability measures
$\mu_k.$ More precisely, we let $\mu_k:=\frac{1}{\kappa_k}\nu_k,$
where $\kappa_k$ is the normalizing constant
$\kappa_k:=\sum\limits_{z\in L_k}\EuScript{L}_k(z)
=\sum\limits_{\underline{x}_1\in\Theta_1^{N_1}}\cdots
\sum\limits_{\underline{x}_k\in\Theta_k^{N_k}}\mathcal{L}(\underline{x}_1)\cdots\mathcal{L}(\underline{x}_k)
=M_1^{N_1}\cdots M_k^{N_k}.$ In order to prove the main result of this article, we
present some lemmas.

\begin{lem}
Suppose $\nu$ is a limit measure of the sequence of probability
measures $\mu_k.$ Then $\nu(F)=1.$
\end{lem}
\begin{proof}
Suppose $\nu=\lim_{k\to\infty}\mu_{l_k}$ for $l_k\to\infty$. For any fixed $l$ and all
$p\geq0$, $\mu_{l+p}(F_l)=1$ since $F_{l+p}\subset F_l$. Thus, $\nu(F_l)\geq\limsup_{k\to\infty}\mu_{l_k}(F_l)=1$.
It follows that $\nu(F)=\lim_{l\to\infty}\nu(F_l)=1$.
\end{proof}

Let $\EuScript{B}=B_n(q,\epsilon/2)$ be an arbitrary ball which
intersects $F.$ Let $k$ be the unique number which satisfies
$t_k\leq n<t_{k+1}.$
Let $j\in\{0,\cdots,N_{k+1}-1\}$ be the unique number so
\begin{align*}
t_k+(n_{k+1}+\max\limits_{x\in\Theta_{k+1}}p(x,n_{k+1},\frac{\epsilon}{4}))j\leq
n<t_k+(n_{k+1}+\max\limits_{x\in\Theta_{k+1}}p(x,n_{k+1},\frac{\epsilon}{4}))(j+1).
\end{align*}
We assume that $j\geq1$ and  the simpler case $j=0$  is similar.

\begin{lem}\label{lem2.11}
For any $p\geq1,$ we have
\begin{align*}
&\mu_{k+p}(\EuScript{B})\\
&\leq\frac{1}{\kappa_kM^j_{k+1}}\exp\Bigg\{S_n\psi(q)+2nVar(\psi,\epsilon)+\|\psi\|(\sum\limits_{i=1}^{k}N_{i}\max\limits_{x\in
\Theta_i}p(x,n_i,\frac{\epsilon}{4})\\&
+j\max\limits_{x\in\Theta_{k+1}}p(x,n_{k+1},\frac{\epsilon}{4})
+n_{k+1}+\max\limits_{x\in\Theta_{k+1}}p(x,n_{k+1},\frac{\epsilon}{4}))\Bigg\}.
\end{align*}
\end{lem}

\begin{proof}
Case $p=1.$ Suppose $\mu_{k+1}(\EuScript{B})>0,$ then $L_{k+1}\cap\EuScript{B}\neq\emptyset$.
Let $z=z(\underline{x},\underline{x}_{k+1})\in L_{k+1}\cap\EuScript{B}$, where $\underline{x}=(\underline{x}_1,\cdots,\underline{x}_k)\in\Theta_{1}^{N_{1}}\times\cdots\times\Theta_{k}^{N_{k}}$
and $\underline{x}_{k+1}\in\Theta_{k+1}^{N_{k+1}}$. Let
\begin{align*}
\EuScript{A}_{\underline{x};x_{1},\cdots,x_{j}}=\{z(\underline{x},y_{1},\cdots,y_{N_{k+1}})\in L_{k+1}:x_{1}=y_{1},\cdots,x_{j}=y_{j}\}.
\end{align*}

Suppose that $z^{\prime}=z(\underline{y},\underline{y}_{k+1})\in L_{k+1}\cap\EuScript{B}$. Since $d_{n}(z,z^{\prime})<\epsilon$,
we have $\underline{y}=\underline{x}$ and $x_{l}=y_{l}$ for $l\in\{1,\cdots,j\}$. We show that $x_l=y_l$ for $l\in\{1,2,\cdots,j\}$
and the proof that $\underline{x}=\underline{y}$ is similar. Suppose that $x_l\neq y_l$ and
let $a_l=t_k+(l-1)(n_{k+1}+\max\limits_{x\in \Theta_{k+1}}p(x,n_{k+1},\frac{\epsilon}{4}))$. Then
\begin{align*}
d_{n_{k+1}}(T^{a_l}z,x_l)<\frac{\epsilon}{4} \text{ and }  d_{n_{k+1}}(T^{a_l}z^\prime,y_l)<\frac{\epsilon}{4}.
\end{align*}
Since $d_{n_{k+1}}(x_l,y_l)>4\epsilon$, we have
\begin{align*}
d_n(z,z^\prime)&\geq d_{n_{k+1}}(T^{a_l}z,T^{a_l}z^\prime)\\
&\geq d_{n_{k+1}}(x_l,y_l)-d_{n_{k+1}}(T^{a_l}z,x_l)-d_{n_{k+1}}(T^{a_l}z^\prime,y_l)>3\epsilon,
\end{align*}
which is a contradicition. Thus we have
\begin{align*}
\nu_{k+1}(\EuScript{B})\leq\sum_{z\in\EuScript{A}_{\underline{x};x_{1},\cdots,x_{j}}}\EuScript{L}_{k+1}(z)=
\mathcal{L}(\underline{x}_1)\cdots\mathcal{L}(\underline{x}_k)\prod\limits_{l=1}^j\exp
S_{n_{k+1}}\psi(x_{l}^{k+1})M_{k+1}^{N_{k+1}-j},
\end{align*}
Case $p>1.$ Similarly,
\begin{align*}
\nu_{k+p}(\EuScript{B})\leq\mathcal{L}(\underline{x}_1)\cdots\mathcal{L}(\underline{x}_k)\prod\limits_{l=1}^j\exp
S_{n_{k+1}}\psi(x_{l}^{k+1})M_{k+1}^{N_{k+1}-j}M_{k+2}^{N_{k+2}}\cdots
M_{k+p}^{N_{k+p}}.
\end{align*}
Since for all $i\in\{1,\cdots,k\}$ and all $l\in\{1,\cdots,N_i\}$,
\begin{align*}
d_{n_i}(T^{t_{i-1}+(l-1)(n_{i}+\max\limits_{x\in \Theta_{i}}p(x,n_{i},\frac{\epsilon}{4}))}z,x_l^i)<\frac{\epsilon}{4},
\end{align*}
we have
\begin{align*}
\mathcal{L}(\underline{x}_1)\cdots\mathcal{L}(\underline{x}_k)\leq\exp\left\{S_{t_k}\psi(z)+t_k\text{Var}(\psi,\epsilon)
+\sum_{i=1}^kN_i\max_{x\in \Theta_{i}}p(x,n_{i},\frac{\epsilon}{4})\|\psi\|\right\}.
\end{align*}
Similarly,
\begin{align*}
\prod_{l=1}^j\exp S_{n_{k+1}}(x_l^{k+1})&\leq\exp\Bigg\{S_{n-t_k}\psi(T^{t_k}z)+(n-t_k)\text{Var}(\psi,\epsilon)\\
&+j\max_{x\in \Theta_{k+1}}p(x,n_{k+1},\frac{\epsilon}{4})\|\psi\|+(n_{k+1}+\max_{x\in \Theta_{k+1}}p(x,n_{k+1},\frac{\epsilon}{4}))\|\psi\|\Bigg\}.
\end{align*}
Combining with the fact $d_n(z,q)<\epsilon$, we obtain
\begin{align*}
&\mathcal{L}(\underline{x}_1)\cdots\mathcal{L}(\underline{x}_k)\prod\limits_{l=1}^j\exp S_{n_{k+1}}\psi({x_{l}^{k+1}})\\
&\leq\exp\Bigg\{S_n\psi(q)+2nVar(\psi,\epsilon)
+\|\psi\|(\sum\limits_{i=1}^{k}N_{i}\max\limits_{x\in
\Theta_i}p(x,n_i,\frac{\epsilon}{4})\\
&+j\max\limits_{x\in\Theta_{k+1}}p(x,n_{k+1},\frac{\epsilon}{4})
+n_{k+1}+\max\limits_{x\in\Theta_{k+1}}p(x,n_{k+1},\frac{\epsilon}{4}))\Bigg\}.
\end{align*}
Since $\mu_{k+p}=\frac{1}{\kappa_{k+p}}\nu_{k+p}$ and $\kappa_{k+p}=\kappa_{k}M_{k+1}^{N_{k+1}}\cdots M_{k+p}^{N_{k+p}}$,
the desired result follows.
\end{proof}

\begin{lem}
For sufficiently large $n,$
\begin{align*}
\limsup\limits_{l\to\infty}\mu_l(B_n(q,\epsilon/2))\leq\exp\Bigg\{-n(\mathbf{C}-2\text{Var}(\psi,\epsilon)-7\gamma)+S_n\psi(q)\Bigg\}.
\end{align*}
\end{lem}
\begin{proof}
From lemma 2.3, we have
\begin{align*}
\kappa_{k}M_{k+1}^j=&M_1^{N_1}\cdots M_k^{N_k}M_{k+1}^j\\
\geq&\exp\{(\mathbf{C}-4\gamma)(N_1n_1+N_2n_2+\cdots+N_kn_k+jn_{k+1})\}\\
=&\exp\{(\mathbf{C}-4\gamma)(N_1(n_1+\max\limits_{x\in\Theta_1}p(x,n_1,\frac{\epsilon}{4}))
+N_2(n_2+\max\limits_{x\in\Theta_2}p(x,n_2,\frac{\epsilon}{4}))+\cdots\\
&+N_k(n_k+\max\limits_{x\in\Theta_k}p(x,n_k,\frac{\epsilon}{4}))+j(n_{k+1}+\max\limits_{x\in\Theta_{k+1}}p(x,n_{k+1},\frac{\epsilon}{4})))-D\}\\
=&\exp\{(\mathbf{C}-4\gamma)n-D-E\},
\end{align*}
where
$$
E=(\mathbf{C}-4\gamma)(n-t_k-j\max\limits_{x\in\Theta_{k+1}}p(x,n_{k+1},\frac{\epsilon}{4})),$$
$$D=(\mathbf{C}-4\gamma)(N_1\max\limits_{x\in\Theta_1}p(x,n_1,\frac{\epsilon}{4}))
+\cdots+N_k\max\limits_{x\in\Theta_k}p(x,n_k,\frac{\epsilon}{4})+j\max\limits_{x\in\Theta_{k+1}}p(x,n_{k+1},\frac{\epsilon}{4})).
$$
Recall that for all $k$, $\max\limits_{x\in\Theta_k}p(x,n_k,\epsilon/4)/n_k<1/2^k$. Hence from (2.3) and (2.4) we obtain for sufficiently large $n$,
$|D/n|<\gamma$, $|E/n|<\gamma$.

By lemma 2.7, for sufficiently large $n$ and any $p\geq1$, we have
\begin{align*}
\mu_{k+p}(\EuScript{B})
&\leq\frac{1}{\kappa_kM_{k+1}^j}\exp\Bigg\{S_{n}\psi(q)+n(2\text{Var}(\psi,\epsilon)+\gamma)\Bigg\}\\
&\leq\exp\Bigg\{-n(\mathbf{C}-2\text{Var}(\psi,\epsilon)-7\gamma)+S_n\psi(q)\Bigg\}.
\end{align*}
Hence the desired results follows.
\end{proof}

Applying the generalized pressure distribution principle, we have
\begin{align*}
P(\widehat{X}(\varphi,T),\psi)\geq P(F,\psi,\epsilon)\geq\mathbf{C}-2\text{Var}(\psi,\epsilon)-7\gamma.
\end{align*}
Recall that $\text{Var}(\psi,\epsilon)<\gamma$, we have
\begin{align*}
P(\widehat{X}(\varphi,T),\psi)\geq P(F,\psi,\epsilon)\geq\mathbf{C}-9\gamma.
\end{align*}
Since $\gamma$ and $\epsilon$ were arbitrary, the proof of theorem 1.2 is complete.
\section{Modification to obtain theorem 1.1}
Fix a small $\gamma>0.$ Let $\mu_1\in E(K,T)$ and satisfes $h_{\mu_1}+\int\psi d\mu_1>\text{\bf C}-\gamma/2.$ Let $\nu\in E(K,T)$ satisfies $\int\varphi d\mu_1\neq
\int\varphi d\nu.$ Let $\mu_2=t_1\mu_1+t_2\nu$ where $t_1+t_2=1$ and $t_1\in(0,1)$ is chosen sufficiently close to 1 so that $h_{\mu_2}+\int\psi d\mu_2>\text{\bf C}-\gamma.$ Obviously, $\int\varphi d\mu_1\neq\int\varphi d\mu_2$. Choose $\delta>0$ sufficiently small so
\begin{align*}
\left|\int\varphi d\mu_1-\int\varphi d\mu_2\right|>4\delta.
\end{align*}
Let $\epsilon>0$. Choose a strictly decreasing sequence $\delta_k\to0$ with $\delta_1<\delta.$ For $k$ odd, we choose a strictly increasing sequence $l_k\to\infty$ so the set
\begin{align*}
Y_k:=\left\{x\in K:\left|\frac{1}{n}S_n\varphi(x)-\int\varphi d\mu_1\right|<\delta_k,
\frac{p(x,n,\epsilon/4)}{n}<\frac{1}{2^k} \text{ for all } n\geq l_k\right\}
\end{align*}
satisfies $\mu_1(Y_k)>1-\gamma$ for every $k.$ For $k$ even, we let $Y_{k,1}:=Y_{k-1}$ and find $l_k>l_{k-1}$ so that each of the sets
\begin{align*}
Y_{k,2}:=\left\{x\in K:\left|\frac{1}{n}S_n\varphi(x)-\int\varphi d\nu\right|<\delta_k,
\frac{p(x,n,\epsilon/4)}{n}<\frac{1}{2^k} \text{ for all } n\geq l_k\right\}
\end{align*}
satisfies $\nu(Y_{k,2})>1-\gamma.$ The proof of the following lemma is similar to that of lemma 2.3.

\begin{lem}
For any small sufficiently $\epsilon>0$ and $k$ even, we can find a sequence $\widehat{n}_k\to\infty$ so $[t_i\widehat{n}_k]\geq l_k$ for $i=1,2$ and sets
$\Theta_k^i$ so that  $\Theta_k^i$  is a $([t_i\widehat{n}_k],4\epsilon)$ separated set for $Y_{k,i}$ with $M_k^i:=\sum\limits_{x\in \Theta_k^i}\exp\Big\{\sum\limits_{j=0}^{[t_{i}\widehat{n}_k]-1}\psi(T^jx)\Big\}$ satisfying
\begin{align*}
M_k^1&\geq\exp\left([t_1\widehat{n}_k](h_{\mu_1}+\int\psi d\mu_1-4\gamma)\right),\\
M_k^2&\geq\exp\left([t_2\widehat{n}_k](h_{\nu}+\int\psi d\nu-4\gamma)\right).
\end{align*}
Then for any $x\in\Theta_{k}^i$, $\frac{p(x,[t_i n_{k}],\epsilon/4)}{[t_i n_{k}]}<\frac{1}{2^{k}}$, where $i=1,2$.
\end{lem}
For $k$ even, let
\begin{align*}
\Theta_k=\Theta_k^1\times\Theta_k^2, M_k=M_k^1 M_k^2 \text{ and }
n_k=[t_1\widehat{n}_k]+\max_{x\in \Theta_k^1}p(x,[t_1\widehat{n}_k],\frac{\epsilon}{4})+[t_2\widehat{n}_k].
\end{align*}
Then $n_k/\widehat{n}_k\to1$.
For $k$ odd, the corresponding ingredients are obtained by lemma 2.3. Given our new construction of $\Theta_k$, the rest of our construction goes through unchanged.
For example, let
$$\underline{x}_1=(x_1^1,x_2^1,\cdots,x_{N_1}^1)\in\Theta_1^{N_1},$$
$$\underline{x}_2=((x_1^{2,1},x_1^{2,2}),(x_2^{2,1},x_2^{2,2}),\cdots,(x_{N_2}^{2,1},x_{N_2}^{2,2}))\in(\Theta_2^1\times\Theta_2^2)^{N_2}.$$
Then for $(\underline{x}_1,\underline{x}_2)\in\Theta_1^{N_1}\times\Theta_2^{N_2}$, by the non-uniform specification property, we have
\begin{align*}
B(\underline{x}_1)=\bigcap_{j=1}^{N_1} T^{-(j-1)(n_{1}+\max\limits_{x\in \Theta_{1}}p(x,n_{1},\frac{\epsilon}{4}))}B_{n_1}(x_j^1,\frac{\epsilon}{4})\neq\emptyset.
\end{align*}
Let $t_1=N_1(n_1+\max_{x\in \Theta_{1}}p(x,n_{1},\frac{\epsilon}{4}))$, then
\begin{align*}
B(\underline{x}_1,\underline{x}_2)=&\bigcap_{j=1}^{N_1} T^{-(j-1)(n_{1}+\max\limits_{x\in \Theta_{1}}p(x,n_{1},\frac{\epsilon}{4}))}
B_{n_1}(x_j^1,\frac{\epsilon}{4})\cap \\
&T^{-t_1}B_{[t_1\widehat{n}_2]}(x_1^{2,1},\frac{\epsilon}{4})\cap
T^{-t_1-[t_1\widehat{n}_2]-\max\limits_{x\in \Theta_2^1}p(x,[t_1\widehat{n}_2],\frac{\epsilon}{4})}B_{[t_2\widehat{n}_2]}(x_1^{2,2},\frac{\epsilon}{4})\cap\cdots\cap\\
&T^{-t_1-(N_2-1)(n_{2}+\max\limits_{x\in \Theta_2^2}p(x,[t_2\widehat{n}_2],\frac{\epsilon}{4}))}B_{[t_1\widehat{n}_2]}(x_{N_2}^{2,1},\frac{\epsilon}{4})\cap\\
&T^{-t_1-(N_2-1)(n_{2}+\max\limits_{x\in \Theta_2^2}p(x,[t_2\widehat{n}_2],\frac{\epsilon}{4}))-[t_1\widehat{n}_2]-
\max\limits_{x\in \Theta_2^1}p(x,[t_1\widehat{n}_2],\frac{\epsilon}{4})}
B_{[t_2\widehat{n}_2]}(x_{N_2}^{2,2},\frac{\epsilon}{4})\neq\emptyset.
\end{align*}

For $k$ even, let $\underline{x}_k=((x_1^{k,1},x_1^{k,2}),(x_2^{k,1},x_2^{k,2}),\cdots,(x_{N_k}^{k,1},x_{N_k}^{k,2}))\in\Theta_k^{N_k}=(\Theta_k^1\times\Theta_k^2)^{N_k}$.
Then for $j\in\{1,\cdots,N_k\}$, we have
\begin{align*}
&\left|S_{[t_1\widehat{n}_k]}\varphi(x_j^{k,1})+S_{[t_2\widehat{n}_k]}\varphi(x_j^{k,2})-n_k\int\varphi d\mu_2\right|\\
\leq&\left|S_{[t_1\widehat{n}_k]}\varphi(x_j^{k,1})+S_{[t_2\widehat{n}_k]}\varphi(x_j^{k,2})-\widehat{n}_k\int\varphi d\mu_2\right|
+\left|\widehat{n}_k\int\varphi d\mu_2-n_k\int\varphi d\mu_2\right|\\
\leq&\left|S_{[t_1\widehat{n}_k]}\varphi(x_j^{k,1})-[t_1\widehat{n}_k]\int\varphi d\mu_1\right|
+\left|S_{[t_2\widehat{n}_k]}\varphi(x_j^{k,2})-[t_2\widehat{n}_k]\int\varphi d\nu\right|+2\|\varphi\|\\
+&\left|\widehat{n}_k\int\varphi d\mu_2-n_k\int\varphi d\mu_2\right|.
\end{align*}
It follows that
\begin{align*}
\left|\frac{S_{[t_1\widehat{n}_k]}\varphi(x_j^{k,1})+S_{[t_2\widehat{n}_k]}\varphi(x_j^{k,2})}{n_k}-\int\varphi d\mu_2\right|\to0.
\end{align*}
Thus we can modify the proof of lemma 2.4 to ensure that our construction still gives rise to points in $\widehat{X}(\varphi,T)$.
Obviously $n_k/\widehat{n}_k\to1$ and $[t_i\widehat{n}_k]/t_i\widehat{n}_k\to1$ for $i=1,2$. Hence for sufficiently large $k$, we have
\begin{align*}
M_k\geq&\exp\{[t_1\widehat{n}_k](h_{\mu_1}+\int\psi d\mu_1-4\gamma)+[t_2\widehat{n}_k](h_\nu+\int\psi d\nu-4\gamma)\}\\
\geq&\exp\{(1-\gamma)\widehat{n}_k(t_1(h_{\mu_1}+\int\psi d\mu_1)+t_2(h_\nu+\int\psi d\mu_1)-4\gamma)\}\\
\geq&\exp(1-\gamma)^2 n_k(h_{\mu_2}+\int\psi d\mu_2-4\gamma)\geq\exp(1-\gamma)^2n_k(\mathbf{C}-5\gamma).
\end{align*}
Our arrival at the second line and the third line is because we are able to add in the extra terms with an arbitrarily small
change to the constant $s$. Since $\gamma$ was arbitrary, we can modify the estimates in lemma 2.8 to cover this more general construction.
\section{Some Applications}
In this section, by the work of Paulo Varandas \cite{Var}, Theorem 1.1 can be applied  to multidimensional local
diffeomorphisms and Viana maps. The BS dimension (introduced by Barreira and  Schmeling) of multifractal decomposition set is also studied in this section.

{\bf Example 1 Mltidimensional local diffeomorphisms} Let $T_0$ be
an expanding map in $\mathbb{T}^n$ and take a periodic point $p$ for
$T_0.$ Let $T$ be a $C^1$-local diffeomorphism obtained from $T_0$
by a bifurcation in a small neighborhood $U$ of $p$ in such a way
that:

(1) every point $x\in\mathbb{T}^n$ has some preimage outside $U$;

(2) $\|DT(x)^{-1}\|\leq\sigma^{-1}$ for every $x\in
\mathbb{T}^n\setminus U,$ and $\|DT(x)^{-1}\|\leq L$ for every
$x\in\mathbb{T}^n$ where $\sigma>1$ is large enough or $L>0$ is
sufficiently close to 1;

(3) $T$ is topologically exact: for every open set $U$ there is
$N\geq1$ for which $T^N(U)=\mathbb{T}^n$.

From \cite{Var} and \cite{VarVia}, we know $T$ has a unique
(ergodic) equilibrium $m$ for H\"{o}lder continuous
potential $-\log|det DT|$ and $(T,m) $ satisfies non-uniform specification property.

\begin{cro}
Let $m$ be the unique ergodic equilibrium state for H\"{o}lder continuous
potential $-\log|det DT|$ in multidimensional local
diffeomorphisms. If $\varphi\in C(\mathbb{T}^n)$ satisfies $\inf\limits_{\mu\in M(K,T)}\int\varphi d\mu
<\sup\limits_{\mu\in M(K,T)}\int\varphi d\mu$, then
\begin{align*}
P(\widehat{\mathbb{T}^n}(\varphi,T),-\log|det DT|)=P(\mathbb{T}^n,-\log|det DT|).
\end{align*}
If $\psi>0,$ then
\begin{align*}
BS(\widehat{\mathbb{T}^n}(\varphi,T),-\log|det DT|)=\sup\left\{h_\nu/\int\psi d\nu:\nu\in M(K,T)\right\}.
\end{align*}
\end{cro}

{\bf Example 2  Viana maps}
In \cite{Via}, the author introduced Viana maps which are obtained as $C^3$ small perturbations of
the skew product $\phi_\alpha$ of the cylinder $S^1\times I$ given by
\begin{align*}
\phi_\alpha(\theta,x)=(d\theta({\rm mod}1),1-ax^2+\alpha
\cos(2\pi\theta))
\end{align*}
where $d\geq16$ is an integer, $a$ is a Misiurewicz parameter for
the quadratic family, and $\alpha$ is small.

Paulo Varandas \cite{Var} proved that when $m$ is SRB measure for $\phi_\alpha$,
then $(\phi_\alpha,m)$ satisfies non-uniform specification.

\begin{cro}
If $m$ is  SRB measure for  a Viana map $\phi_\alpha$ and $\varphi\in C(S^1\times I)$ satisfies $\inf\limits_{\mu\in M(K,\phi_\alpha)}\int\varphi d\mu
<\sup\limits_{\mu\in M(K,\phi_\alpha)}\int\varphi d\mu$, then for any $\psi\in C(S^1\times I)$
\begin{align*}
P(\widehat{S^1\times I}(\varphi,\phi_\alpha),\psi)\geq&\sup\left\{h_\nu+\int\psi d\nu:\nu\in M(K,\phi_\alpha)\right\}.
\end{align*}
If $\psi>0,$ then
\begin{align*}
BS(\widehat{S^1\times I}(\varphi,\phi_\alpha),\psi)\geq&\sup\left\{h_\nu/\int\psi d\nu:\nu\in M(K,\phi_\alpha)\right\}.
\end{align*}
If there exists $\mu\in M(K,T)$ such that $\mu$ is a equilibrium state for $\phi_\alpha$ with respect to potential $\psi\in C(X)$, then we have
$P(\widehat{S^1\times I}(\varphi,\phi_\alpha),\psi)=P(S^1\times I,\psi)$.
\end{cro}

\noindent {\bf Acknowledgements.}   The authors would like to thank the anonymous referees for carefully reading our paper and providing suggestions
 for improvement. The work was supported by the
National Natural Science Foundation of China (grant No. 11271191)
and National Basic Research Program of China (grant No.
2013CB834100) and  the Fundamental Research Funds for the Central Universities (grant No. WK0010000035).

\end{document}